\documentclass[11pt,reqno]{amsart}

\usepackage{amsmath,color}
\usepackage{amsthm}
\usepackage{amssymb}
\usepackage[all]{xy}
\newcounter{obs}

\newtheorem{theorem}{Theorem}[section]
\newtheorem{proposition}[theorem]{Proposition}
\newtheorem{corollary}[theorem]{Corollary}

\theoremstyle{definition}
\newtheorem{definition}[theorem]{Definition}
\newtheorem{example}[theorem]{Example}

\newtheorem{remark}[theorem]{Remark}

\title[Domination spaces and factorization of operators]{Domination spaces and factorization of \\ linear and multilinear summing operators}

\author[E. Dahia, D. Achour, P. Rueda and E. A. S\'{a}nchez P\'{e}rez]
{E. Dahia, D. Achour, P. Rueda and E. A. S\'{a}nchez P\'{e}rez
}

\address{E. Dahia\\
Universit\'{e} Mohamed Boudiaf-M'sila , Laboratoire d'Analyse Fonctionnelle et G\'{e}om\'{e}trie des Espaces\\
28000 M'sila,
Algeria.}

\email{hajdahia@gmail.com}

\address{D. Achour\\
Universit\'{e} Mohamed Boudiaf-M'sila, Laboratoire d'Analyse Fonctionnelle et G\'{e}om\'{e}trie des Espaces\\
28000 M'sila,
Algeria.}
\email{dachourdz@yahoo.fr}

\address{P. Rueda\\
Departamento de An\'{a}lisis Matem\'{a}tico\\ Universidad de Valencia, 46100 Burjasot
Valencia. Spain.}

\email{pilar.rueda@uv.es}

\address{E. A. S\'{a}nchez P\'{e}rez\\
Instituto Universitario de Matem\'atica Pura y Aplicada,
Universitat Politecnica de Valencia\\ Camino de Vera s/n, 46022
Valencia. Spain.}

\email{easancpe@mat.upv.es}


\keywords{Pietsch's domination
theorem, factorization of operators, multilinear operators
\\
\emph{Corresponding Author}. E. DAHIA, Laboratoire d'Analyse Fonctionnelle et G\'{e}om\'{e}trie des Espaces,
Universit\'{e} Mohamed Boudiaf-M'sila (Algeria). hajdahia@gmail.com
\\
The third author was supported  by the
Ministerio de Econom\'{\i}a y Competitividad (Spain) under grant
 MTM2011-22417.
The fourth  author was supported by the
Ministerio de Econom\'{\i}a y Competitividad (Spain) under grant
MTM2012-36740-C02-02.
}

\begin{document}
\maketitle

\begin{abstract}
It is well known that not every summability property for multilinear operators leads to a factorization theorem. In this paper we undertake a detailed study of factorization schemes for summing linear and nonlinear operators. Our aim is to integrate under the same theory a wide family of  classes of mappings for which a Pietsch type factorization theorem holds.
Our construction includes the cases of absolutely $p$-summing linear operators, $(p,\sigma)$-absolutely continuous linear operators, factorable strongly $p$-summing multilinear operators, $(p_1,\ldots,p_n)$-dominated multilinear operators and dominated $(p_1,\ldots, p_n;\sigma)$-continuous multilinear operators.
\end{abstract}

\section{Introduction and preliminaries}

Domination and factorization properties of linear and non linear operators are usually restricted to expressions that involve the $p$-homogeneous scalar map $t \rightsquigarrow |t|^p$ and its inverse. Actually, even the most abstract approaches to the problem are based in this kind of expressions (see \cite{BPR10, PeSa,PeSaSe}). In \cite{BPR10} the $R$--$S$-abstract $p$-summing mappings were introduced with the aim of unifying the wide variety of summing classes of mappings that can be found in the literature. These abstract approaches provide  domination results that recover all the   known related dominations for classes of summing mappings in the settings of linear and multilinear operators, homogeneous polynomials, Lipschitz mappings and subhomogeneous mappings among others.  However, the main achievements to find abstract domination theorems inspired in absolutely summing operators are not accompanied by  general factorization theorems. It is  known that although for some  classes  of multilinear mappings a domination theorem holds, it is sometimes not clear whether this can be written as a standard factorization  (see \cite{Dim03, PeRuSa}). Some progress was made in \cite{PeRuSa,RuSa}, where the tandem domination/factorization for classes of summing polynomials and multilinear operators was orchestrated. In this paper we will show  a meaningful  set of vector norm inequalities that provide not only domination theorems, \textit{but also} the corresponding Pietsch factorization schemes. To do so, we introduce a general family of summability properties for linear and multilinear mappings, that are characterized by means of domination inequalities that \textit{always} lead to factorization theorems. The related   class of summing operators  is defined by means of a homogeneous mapping $\Phi$, and can be seen as a  subclass of $R$--$S$-abstract  $p$-summing mappings in the sense of \cite{BPR10}, that includes the classes of absolutely $p$-summing linear operators, $(p,\sigma)$-absolutely continuous linear operators, factorable strongly $p$-summing multilinear operators, $(p_1,\ldots,p_n)$-dominated multilinear operators and dominated $(p_1,\ldots, p_n;\sigma)$-continuous multilinear operators.
 Therefore, our aim is to study these domination properties related to $\Phi$,  and to establish the truth  of the relations
$$
\textit{Summability inequality} \, \Leftrightarrow \, \textit{Operator domination} \, \Leftrightarrow \, \textit{Operator factorization}
$$
for all these classes.

This paper is organized as follows. After this introduction and some
preliminaries, in Section 2 we present the factorization theorem through a
domination space for\ the class of linear operators which are defined by a
summability property involving a homogeneous map: the $\Phi$-abstract $p$-summing operators. We use techniques
inspired in the convexification of a domination inequality. In the third
section we find the factorization of two
categories of multilinear maps  by using suitable domination spaces: the strongly $\Phi$-abstract $p$-summing and the  $(\Phi _{1},...,\Phi _{m})$-abstract $(p_{1},...,p_{m})$-summing multilinear operators.
Finally, in the last section we study a multilinear
version of $(p,\sigma )$-absolutely continuous linear operators in order to
apply the technique of factorization  previously developed.

Let $m\in \mathbb{N}$ and $X_{1},...,X_{m},Y$
be Banach spaces over $\mathbb{K}$, (either $\mathbb{R}$ or $\mathbb{C}$). We will denote by $\mathcal{L}\left(
X_{1},...,X_{m};Y\right) $ the Banach space of all continuous $m$-linear
mappings from $X_{1}\times ...\times X_{m}$ to $Y$, under the norm
\begin{equation*}
\left\Vert T\right\Vert =\underset{x^{_{j}}\in B_{X_{j}}}{\sup }\left\Vert
T(x^{1},...,x^{m})\right\Vert ,
\end{equation*}
where $B_{X_{j}}$ denotes the closed unit ball of $X_{j}$ $(j=1,...,m).$ If $Y=\mathbb{K}$, we write $\mathcal{L}\left( X_{1},...,X_{m}\right).$ In
the case $X_{1}=...=X_{m}=X$, we will simply write $\mathcal{L}%
\left(^{m}X;Y\right),$ whereas $\mathcal L(X,Y)$ is the usual Banach space of all continuous linear operators from $X$ to $Y$.

By $X_{1}\widehat{\otimes }_{\pi }\cdots\widehat{\otimes }_{\pi }X_{m}$ we
denote the projective tensor product of $X_{1},\ldots,X_{m}.$ The
projective norm is defined by
\begin{equation*}
\pi (v)=\inf \left\{ \overset{n}{\underset{i=1}{\sum }}\overset{m}{\underset{
j=1}{\prod }}\left\Vert x_{i}^{j}\right\Vert ,n\in \mathbb{N},v=\overset{n}{
\underset{i=1}{\sum }}x_{i}^{1}\otimes \cdots\otimes x_{i}^{m}\right\}.
\end{equation*}

We consider the canonical continuous multilinear mapping
\begin{equation*}
\sigma _{m}:X_{1}\times \ldots \times X_{m}\longrightarrow X_{1}\widehat{%
\otimes }_{\pi }\cdots \widehat{\otimes }_{\pi }X_{m}
\end{equation*}%
defined by%
\begin{equation*}
\sigma _{m}(x^{1},...,x^{m})=x^{1}\otimes \cdots \otimes x^{m}
\end{equation*}
 for
all $(x^{1},\ldots ,x^{m})\in X_{1}\times \cdots \times X_{m}.$
Given $T\in \mathcal{L}(X_{1},\ldots ,X_{m};Y)$, consider its linearization $%
T_{L}:X_{1}\widehat{\otimes }_{\pi }\cdots \widehat{\otimes }_{\pi
}X_{m}\rightarrow Y,$ given by $T_{L}(x^{1}\otimes \cdots \otimes
x^{m})=T\left( x^{1},\ldots ,x^{m}\right) $ and extended by linearity.  It is well
known that $T=T_{L}\circ \sigma _{m}$ and the Banach space $\mathcal{L}%
(X_{1},\ldots ,X_{m};Y)$ is isometrically isomorphic to $ \mathcal{L}%
(X_{1}\widehat{\otimes }_{\pi }\cdots \widehat{\otimes }_{\pi }X_{m},Y)\ $%
through the correspondence$\ T_{L}\longleftrightarrow T.$\ For the general
theory of tensor products we refer to \cite{DF92,Rya02}.

Let $1\leq p< \infty $.
Let $\ell _{p}\left( X\right) $ be the Banach space of all absolutely $p$%
-summable sequences $\left( x_{i}\right) _{i=1}^{\infty }$ in the Banach space $X$ with the
norm
\begin{equation*}
\left\Vert \left( x_{i}\right) _{i=1}^{\infty }\right\Vert
_{p}=(\sum\limits_{i=1}^{\infty }\left\Vert x_{i}\right\Vert ^{p})^{\frac{1}{%
p}}.
\end{equation*}%
We denote by $\ell _{p,\omega }\left( X\right) $ the Banach space of all
weakly $p$-summable sequences $\left( x_{i}\right) _{i=1}^{\infty }$ in $X$
with the norm
\begin{equation*}
\left\Vert \left( x_{i}\right) _{i=1}^{\infty }\right\Vert _{p,\omega }=%
\underset{x^{\ast }\in B_{X^{\ast }}}{\sup }\left\Vert \left( \langle
x_{i},x^{\ast }\rangle \right) _{i=1}^{\infty }\right\Vert _{p}.
\end{equation*}%
Note that $\ell _{p,\omega }\left( X\right) =\ell _{p}\left( X\right) $ for
some $1\leq p<\infty $ if, and only if, $X$ is finite dimensional.

Let $X$, $Y$ and $E$ be (arbitrary) sets, $\mathcal{H}$ be a family of
mappings from $X$ to $Y$, $G$ be a Banach space and $K$ be a compact Hausdorff
topological space. Let
\[
R \colon K\times E \times G\longrightarrow\lbrack0,\infty)~\mathrm{and~} S
\colon{\mathcal{H}}\times E \times G\longrightarrow\lbrack0,\infty)
\]
be arbitrary mappings and $0<p<\infty$. A mapping
$f\in{\mathcal{H}}$ is said to be \textit{$R$-$S$-abstract $p$-summing} is
there is a constant $C>0$ so that%
\begin{equation}
\left(\sum_{j=1}^{m}S(f,x_{j},b_{j})^{p}\right)^{\frac1p}\leq C\sup_{\varphi\in K}\left(\sum_{j=1}%
^{m}R\left(  \varphi,x_{j},b_{j}\right)  ^{p}\right)^{\frac1p},\label{cam-errado}%
\end{equation}
for all $x_{1},\ldots,x_{m}\in E,$ $b_{1},\ldots,b_{m}\in G$ and
$m\in\mathbb{N}$.

Throughout all the paper we will say that a function $\Phi$ between Banach spaces $X$ and $Y$ is \textit{homogeneous} if for every positive real number $\lambda$, we have that for every $x \in X$, $\Phi(\lambda x)= \lambda \Phi(x)$. This is sometimes called to be positive homogeneous. We will say that it is \textit{bounded} if there is a constant $K$ such that for every $x \in X$, $\| \Phi(x)\| \le K \|x\|$.

To finish this introductory section, let us provide the definition and main properties on Banach function spaces on finite measure spaces and their $p$-th powers. Our main references are Chapter 2  (in particular Section 2.2) in \cite{oksan} and \cite[pp.28,51]{LiTz}.
 Let $(\Omega,\Sigma,\mu)$ be a complete finite measure space.
A real Banach space
$Z(\mu)$ of (equivalence classes of) $\mu$-measurable functions is a {\it Banach function space}  over $\mu$ if $Z(\mu) \subset L^1( \mu)$, it
contains all the simple functions and, if
 $\Vert\cdot\Vert_{Z(\mu)}$ is the norm of the space,  $g\in Z(\mu)$ and $f$ is a measurable function such that  $|f|\le|g|$ $\mu$--a.e., then $f\in
Z(\mu)$ and $\Vert f\Vert_{Z(\mu)}\le\Vert g\Vert_{Z(\mu)}$ (\cite[Def.1.b.17,
p.28]{LiTz}).
 We always have that
 $L^\infty(\mu)\subset Z(\mu)\subset L^1(\mu)$. 


Let $1 \le p < \infty$.
The $p$-th power  of $Z(\mu)$  is
defined as the set of functions
$$
Z(\mu)_{[p]}:=\{f \in L^0(\mu): |f|^{1/p} \in Z(\mu)\}
$$
that is a Banach function space over $\mu$  with
the norm $\|f\|_{Z(\mu)_{[p]}}:= \| |f|^{1/p} \|_{Z(\mu)}^p$, $f \in
Z(\mu)_{[p]}$ whenever  if $Z(\mu)$ is $p$-convex (with $p$-convexity constant equal to $1$). 
 If the measure $\mu$ is clearly fixed in the context, we will write $Z$ and $Z_{[p]}$ instead of $Z(\mu)$ and $Z(\mu)_{[p]}$ respectively.

We will use also some properties of the space $C(K)$ of all continuous functions on a compact Hausdorff topological space $K$.
Let $S$ be a subspace of  $C(K)$ and let $\mu$ be a probability Borel measure on $K$. Consider a Banach function space $Z$ over $\mu$. We will say that the identification map $S \to Z$ is well-defined if for every $f \in S$, the map that identifies $f$ with the equivalence class $[f] \in Z$ of all  functions that are $\mu$-a.e. equal to $f$, is well-defined and continuous. When we write $C(B_{X^*})$ for some Banach space $X$, $B_{X^*}$ is considered as a  compact space with the weak* topology.

\section{Factorization of $\Phi$-abstract $p$-summing linear operators}
\subsection{$\Phi$-abstract $p$-summing linear operators}

Recall that each Banach space $X$ can be considered (isometrically) as a subspace
$\widehat{X}$ of $C(B_{X^{\ast }})$ just by considering the elements $x$ of $X $ as functions $\langle x,\cdot \rangle: X^* \to \mathbb R$ defined as $\langle x,\cdot \rangle(x^*):= \langle x, x^*\rangle$, $x^* \in X^*$. Of course, here $B_{X^{\ast }}$ is considered as a compact subset with respect to the weak* topology, and so the elements of $\widehat{X}$ are continuous functions. Throughout the paper we will write
$i$ for the isometric embedding $i:X \to \widehat X \subset C(B_{X^{\ast }})$ given by $i(x)=\langle x, \cdot \rangle$, $x \in X$. Let $\Phi :\widehat{X} \rightarrow C(B_{X^{\ast }})$ be a homogeneous function
and let $1\leq p<\infty $.

\begin{definition} Let $X$ and $Y$ be Banach spaces.
A linear operator $T:X\longrightarrow Y$  is {\it $\Phi$-abstract $p$-summing} if there is a constant $C>0$ such that for each $%
x_{1},...,x_{n}\in X$\ and all $n\in \mathbb{N},$ we have%
\begin{equation*}
\left( \sum_{i=1}^{n}\left\Vert T(x_{i})\right\Vert ^{p}\right) ^{1/p}\leq
C\left\Vert \left( \sum_{i=1}^{n}\left\vert \Phi (\left\langle x_{i},\cdot
\right\rangle )\right\vert ^{p}\right) ^{\frac{1}{p}}\right\Vert
_{C(B_{X^{\ast }})}.
\end{equation*}
\end{definition}

Notice that if $\Phi $ is the inclusion map $\widehat{X}\hookrightarrow
C(B_{X^{\ast }})$, the definition gives the space of absolutely $p$-summing
operators. Let us develop in what follows some relevant examples. \newline

\begin{itemize}
\item[(1)] Matter \cite{Matt87} introduced the concept of $(p,\sigma )$-absolutely continuous linear operators $(1\leq
p<\infty ,0\leq \sigma <1)$, that was
developed later by L\'{o}pez Molina and the fourth author \cite{LS93} and extended in \cite{adsp14,YaAcRu}. A linear
operator $T:X\to Y$ is {\it $(p,\sigma )$-absolutely
continuous} if there is a positive
constant $C$ such that for all $n\in \mathbb{N}$ and $\left( x_{i}\right)
_{i=1}^{n}\subset X,$ we have
\begin{equation}
\left( \underset{i=1}{\overset{n}{\sum }}\left\Vert T(x_{i})\right\Vert ^{%
\frac{p}{1-\sigma }}\right) ^{\frac{1-\sigma }{p}}\leq C\underset{x^{\ast
}\in B_{X^{\ast }}}{\sup }\left( \underset{i=1}{\overset{n}{\sum }}\left(
\left\vert \left\langle x_{i},x^{\ast }\right\rangle \right\vert ^{1-\sigma
}\left\Vert x_{i}\right\Vert ^{\sigma }\right) ^{\frac{p}{1-\sigma }}\right)
^{\frac{1-\sigma }{p}} .  \label{Linabsc}
\end{equation}

\noindent  $(p,\sigma )$-Absolutely continuous
operators are $\Phi$-abstract $p$-summing  when  the function $\Phi $ is defined by $%
\Phi (\langle x,\cdot \rangle ):=\Vert x\Vert ^{\sigma }|\langle x,\cdot
\rangle |^{1-\sigma }$, $x\in X$. In particular, absolutely $p$-summing operators are recovered when $\sigma=0$. \newline

\item[(2)] Consider the function $\Phi :X\longrightarrow C(B_{X^{\ast
}})$ given by $\Phi (\langle x,\cdot \rangle ):=\frac{|\langle
x,\cdot \rangle |^{2}}{\Vert x\Vert },$ $x\in X.$ Take also $p=1$ and a linear operator $%
T:X\longrightarrow Y$. The summability inequality can be written as
\begin{equation*}
\sum_{i=1}^{n}\Vert T(x_{i})\Vert \leq C\left\Vert \sum_{i=1}^{n}\frac{%
|\langle x_{i},\cdot \rangle |^{2}}{\Vert x_{i}\Vert }\right\Vert
_{C(B_{X^{\ast }})}
\end{equation*}%
for each $x_{1},...,x_{n}\in X$. Note that an operator satisfying
such a domination property is always $1$-summing, since $\frac{|\langle
x,\cdot \rangle |^{2}}{\Vert x\Vert }\leq |\langle x,\cdot \rangle |$ for
all $x\in X$.
\newline

\item[(3)] Take $p=1$. Fix a norm one functional $x_{0}^{\ast }\in X^{\ast }$
and consider the function
\begin{equation*}
\Phi (\langle x,\cdot \rangle ):=|\langle x,x_{0}^{\ast }\rangle
|^{1/2}|\langle x,\cdot \rangle |^{1/2}.
\end{equation*}%
Note that, if $T$ satisfies such a domination, we obtain that

\begin{equation*}
\sum_{i=1}^{n}\Vert T(x_{i})\Vert \leq C\left\Vert \sum_{i=1}^{n}|\langle
x_{i},x_{0}^{\ast }\rangle |^{\frac{1}{2}}|\langle x_{i},\cdot \rangle |^{%
\frac{1}{2}}\right\Vert _{C(B_{X^{\ast }})}.
\end{equation*}

Applying \cite[Theorem 2.1]{PeSa} to this class of mappings,
 (see the explanation at the end of this subsection)  we
 find a regular Borel probability measure $\eta $ on $B_{X^{\ast }}$ (with the weak* topology) such that
\begin{equation*}
\Vert T(x)\Vert \leq C|\langle x,x_{0}^{\ast }\rangle
|^{1/2}\int_{B_{X^{\ast }}}|\langle x,\cdot \rangle |^{1/2}\,d\eta \leq
\frac{1}{2}C\big(|\langle x,x_{0}^{\ast }\rangle |+\int_{B_{X^{\ast
}}}|\langle x,\cdot \rangle |\,d\eta \big)
\end{equation*}%
\begin{equation*}
=C\int_{B_{X^{\ast }}}|\langle x,\cdot \rangle |\,d\big(\frac{\delta
_{x_{0}^{\ast }}+\eta }{2}\big).
\end{equation*}
where $\delta _{x_{0}^{\ast }}$ is the Dirac's delta associated with $
x_{0}^{\ast }\in B_{X^{\ast }}$.
In particular, this shows that this class is contained in the class on
absolutely summing operators.
\end{itemize}

The $\Phi$-abstract $p$-summing  operators form a subclass of $R$--$S$-abstract $p$-summing mappings. To see this, just consider Banach spaces $X$ and $Y$, $E=X$, $\mathcal H=\mathcal L(X,Y)$, $G=\mathbb R$, $K=B_{X^*}$ endowed with the weak* topology, and define
$$
R:B_{X^*}\times X \times \mathbb R\to [0,\infty) \ \mbox{ and } \ S:\mathcal L(X,Y)\times X\times \mathbb R\to [0,\infty)
$$
by $R(\phi,x,b):=|\Phi(\langle x,\phi\rangle )||b|$ and $S(T,x,b):=\|T(x)\| |b|$.

An application of \cite[Theorem 2.2]{BPR10} gives that an operator $T:X\to Y$ is $\Phi$-abstract $p$-summing if and only if there exists a constant $C>0$ and a regular Borel probability measure $\mu$ on $B_{X^*}$ such that
$$
\|T(x)\|  \leq C\left(\int_{B_{X^*}} |\Phi(\langle x,x^*\rangle )|^p d\mu(x^*)\right)^{1/p}
$$
 for all $x\in X$.

\subsection{Domination spaces}

Let $\Phi :\widehat{X}\rightarrow C(B_{X^{\ast }})$ be a
homogeneous function. Let $Z$ be
a Banach space and suppose that there is a (norm one) continuous
linear operator $j:span\{\Phi(\widehat{X}) \}\rightarrow Z$. We define on $\widehat{X}$ a seminorm associated to $\Phi$ and $Z$ as
$$
\|\langle x,\cdot\rangle \|_{\Phi,Z}:= \inf  \sum_{i=1}^r \|j\circ \Phi(\langle x_i, \cdot \rangle)\|_Z,
$$
where the infimum is computed over all decompositions of $x$ as $x= \sum_{i=1}^r x_i$, $x, x_i \in X$. In what follows we will  assume always that this seminorm is continuous with respect to the norm of
$\widehat{X}$.

Consider the subspace $N:=\{\langle x,\cdot \rangle\in \widehat{X}: \| \langle x,\cdot \rangle \|_{\Phi,Z}=0\} \subseteq \widehat{X}$.
The completion of the corresponding quotient space $\widehat{X}/N$ is what we call the {\it domination space $\widehat{X}_{\Phi}^Z$.}

Let us consider now the quotient map $i_{\widehat{X}_{\Phi}^Z}: \widehat X \to \widehat{X}_{\Phi}^Z$, that is defined as $i_{\widehat{X}_{\Phi}^Z}( \langle x,\cdot \rangle):= [\langle x,\cdot \rangle]$, $\langle x, \cdot \rangle \in \widehat{X}$, where $[ \langle x,\cdot \rangle ]$ denotes the equivalence class of the continuous function $\langle x,\cdot \rangle$ in $\widehat{X}/N$. Under the assumption  that $\Phi$ is bounded, the map  $i_{\widehat{X}_{\Phi}^Z}$ is continuous whenever we consider $\widehat X $ endowed with the supremum norm induced by $C(B_{X^*})$. Indeed, if $x \in X$ then
\begin{eqnarray*}
\|i_{\widehat{X}_{\Phi}^Z}(\langle x,\cdot\rangle)\|_{\widehat{X}_{\Phi}^Z}&=& \|[\langle x,\cdot\rangle]\|_{\widehat{X}_{\Phi}^Z}=  \|\langle x,\cdot\rangle\|_{\Phi,Z}\\
&\leq & \|j\circ \Phi (\langle x,\cdot\rangle )\|_Z \leq \|\Phi(\langle x,\cdot\rangle )\|_{C(B_{X^*})}\\
&\leq & K \|\langle x,\cdot\rangle \|_{C(B_{X^*})},\\
\end{eqnarray*}
where $K$ is the constant that comes from the boundedness of $\Phi$.

Let $T:X \to Y$ be a continuous linear  operator. Having in mind the classical Pietsch factorization scheme for $p$-summing operators,  the construction of the space $\widehat{X}_{\Phi}^Z$
aims to define a factorization scheme for $T$ associated to $Z$ and $\Phi$ as the one that follows, provided that some suitable  inequality holds. Indeed, we will show that the diagram
$$
\xymatrix{
X \ar[rr]^{T} \ar@{>}[d]_{i} & & Y\\
\widehat X \ar@{>}[rr]^{i_{\widehat{X}_{\Phi}^Z}} &  & \widehat{X}_{\Phi}^Z  \ar@{>}[u]_{\widehat T}}
$$
is commutative whenever the operator $T$ satisfies an adequate inequality that assures the continuity of $\widehat T$.

The main example of this kind of factorization is the classical diagram that holds for $p$-summing operators.  The classical Pietsch domination theorem asserts that an operator $T:X\to Y$ is absolutely $p$-summing if, and only if, there exists a constant $C>0$ such that
$$
\|T(x)\|\leq C\left( \int_{B_{X^*}}|\langle x, x^* \rangle|^p d\mu(x^*)\right)^{1/p}
$$
for all $x\in X$. This domination by the $L^p$-norm allows to achieve the well-known factorization theorem for $p$-summing operators through a subspace of $L^p(\mu)$ for the Pietsch measure $\mu$. In general, it is expected that the existence of a domination by a norm of a certain  space provides a factorization theorem. Our aim is to show that this is the case for $\Phi$-abstract $p$-summing operators. In the multilinear case ---as it will be shown in the next section---  a domination like this does not lead necessarily to a factorization of the multilinear map.

Let us provide other examples of summing inequalities that in fact lead to  $p$-summing type factorization schemes. Following the standard definition, the operator $i_{\widehat{X}_{\Phi}^Z}: \widehat X  \to \widehat{X}_{\Phi}^Z$ is $p$-concave if there  is $C>0$ such that for $\langle x_1, \cdot \rangle,...,\langle x_n, \cdot \rangle \in \widehat X$,

\[
\left( \sum_{i=1}^{n}\Vert \langle x_{i},\cdot \rangle \Vert _{\widehat{X}_{\Phi}^Z}^{p}\right) ^{1/p}\leq C\left\Vert \left( \sum_{i=1}^{n}|\langle
x_{i},\cdot \rangle |^{p}\right) ^{\frac{1}{p}}\right\Vert _{C(B_{X^{\ast
}})}.
\]

The following result provides a characterization of $p$-summing operators using our tools.

\begin{proposition} Let $T:X \to Y$ be an operator and  let $1 \le p < \infty.$
The following statements  are equivalent.
\begin{itemize}

\item[(1)] There is a homogeneous  map $\Phi: \widehat X \to C(B_{X^*})$, a Banach lattice  $Z$  such that $i_{\widehat{X}_{\Phi}^Z}$ is $p$-concave and a continuous linear operator $\widehat T: \widehat{X}_{\Phi}^Z \to Y$ such that
$T= \widehat T \circ i_{\widehat{X}_{\Phi}^Z} \circ i.$

\item[(2)]  $T$ is $p$-summing.
\end{itemize}
\end{proposition}
\begin{proof}
To prove  (1) $\Rightarrow$ (2), note that $p$-concavity of $i_{\widehat{X}_{\Phi}^Z}$ implies that it is $p$-summing. For the converse, just notice that the Pietsch factorization of the $p$-summing operator gives (1) for $Z= L^p(\eta)$ for a certain Borel probability measure $\eta$ and $\Phi$ just the identity.
\end{proof}

This easy characterization opens the door to easy sufficient conditions that come from our abstract factorization diagram for $T$ to be $p$-summing. Let $T$ be an operator that factors through $\widehat{X}_{\Phi}^Z$ as above.

\begin{itemize}
\item {If $\widehat{X}_{\Phi}^Z$ is $p$-concave then $T$ is $p$-summing, since $ i_{\widehat{X}_{\Phi}^Z}$ is a positive operator (see \cite{LiTz}). }

\item {If $Z$ is $p$-concave and $|\Phi ( \langle x, \cdot \rangle)| \le C |\langle x, \cdot \rangle|$ for all $x \in X$ and a constant $C>0$, then $T$ is $p$-summing.}

\item  {If $Z$ is $p$-concave and $ \|( \sum_{i=1}^n |\Phi ( \langle x_i, \cdot \rangle)|^p )\|^{1/p} \le C
 \|( \sum_{i=1}^n  |\langle x_i, \cdot \rangle|^p )^{1/p}\|$
for each finite family $x_1,...,x_n \in X$, then $T$ is $p$-summing.}

\end{itemize}

The next result gives the domination characterization of our family of summing operators.

\begin{proposition} \label{1}
Let $T:X \to Y$ be an operator between Banach spaces and let
$\Phi: \hat X \to C(B_{X^*})$ be a   homogeneous map. Let be
\begin{itemize}
\item[(i)] either $Z=C(B_{X^*})$,

\item[(ii)] or
  $Z=Z(\mu)$,  a $p$-convex Banach function space over $\mu$  ---where  $\mu$ is a Borel probability measure on $B_{X^*}$--- such that the inclusion  map $C(B_{X^*}) \to Z$ is well-defined.
  \end{itemize}
The following assertions are equivalent.

\begin{itemize}

\item[(1)] For each finite set $x_1,...,x_n \in X,$
$$
\left( \sum_{i=1}^n \| T(x_i)\|^p \right)^{1/p} \le \left\| \left( \sum_{i=1}^n | \Phi( \langle x_i, \cdot \rangle)|^p \right)^{1/p} \right\|_{Z}.
$$

\item[(2)] There is a positive functional $z^*_p$ in the topological  dual of the $p$-th power of $Z$ such that for each  $x \in X,$
$$
\|T(x)\| \le \big\langle |\Phi(\langle x, \cdot \rangle)|^p, z^*_p  \big\rangle^{1/p}.
$$

\end{itemize}

\end{proposition}
\begin{proof}
To prove  (1) $\Rightarrow$ (2) use the fact that each $p$-convex Banach function space has the property that its $p$-th power $Z_{[p]}$  is (maybe after renorming) a Banach function space. If $Z=C(B_{X^*})$, then simply notice that the $p$-th power of $Z$ coincides with $Z$. We may assume w.l.o.g. that the $p$-convexity constant of $Z$ is $1$, and so the quasi-norm given in the definition of the $p$-th power is a norm. Furthermore,
$$
\left\|\left( \sum_{i=1}^n | \Phi( \langle x_i, \cdot \rangle)|^p \right)^{1/p} \right\|_{Z}=\left\| \sum_{i=1}^n |\Phi(\langle x_i, \cdot \rangle)|^p \right\|_{Z_{[p]}}^{1/p},
$$
for each finite set $x_1,\ldots,x_n \in X$, and then the inequality in (1) can be written as
$$
 \sum_{i=1}^n \| T(x_i)\|^p  - \sup_{g \in B_{(Z_{[p]})^*}} \left\langle  \sum_{i=1}^n |\Phi( \langle x_i, \cdot \rangle)|^p , g  \right\rangle \le 0.
$$
Consider all functions given by
$$
\psi(g):=
 \sum_{i=1}^n \| T(x_i)\|^p   -  \left\langle  \sum_{i=1}^n |\Phi( \langle x_i, \cdot \rangle)|^p , g  \right\rangle,
$$
$g \in B_{(Z_{[p]})^*}.$
A standard argument using the homogeneity of $\Phi$ and Ky Fan's Lemma (see, e.g. \cite[Lemma 6.12]{oksan}) for the concave family of all these convex w*-continuous functions gives the result. The converse is proved by a direct computation.
\end{proof}

Note that for $0 \le z^*_p \in (Z_{[p]})^*$, the expression  $f \mapsto \big\langle|f|^p, z^*_p  \big\rangle^{1/p}$, $f \in C(B_{X^*})$ defines an abstract $p$-seminorm, that can be used together with $\Phi$ to define an ``$L^p$-type"  domination space. The next result shows that in fact this is the canonical domination space, and any other domination defined by another Banach function space $Z$ leads to a factorization through an $L^p$-type domination space. We show in the next result that
 the requirement on the existence of a continuous identification map $C(B_{X^*}) \to Z$ always implies such a factorization.


\begin{theorem} \label{pcoro}
Let $T:X \to Y$ be an operator between Banach spaces and let
$\Phi$ a bounded   homogeneous map $\Phi: \widehat X \to C(B_{X^*})$.
The following assertions are equivalent.

\begin{itemize}

\item[(1)] There is a  constant $K>0$ and a Banach function space $Z$ over a finite Borel measure $\mu$ on $B_{X^*}$ such that the identification  map $C(B_{X^*}) \to Z$ is well-defined (and so continuous), and for each finite set $x_1,...,x_n \in X,$
$$
\left( \sum_{i=1}^n \| T(x_i)\|^p \right)^{1/p} \le K \left\| \left( \sum_{i=1}^n |\Phi(\langle x_i, \cdot \rangle)|^p \right)^{1/p} \right\|_{Z}.
$$

\item[(2)] The operator $T$ is $\Phi$-abstract $p$-summing, that is, there is a constant $K>0$ such that  for each finite set $x_1,...,x_n \in X,$
$$
\left( \sum_{i=1}^n \| T(x_i)\|^p \right)^{1/p} \le K\left\| \left( \sum_{i=1}^n |\Phi(\langle x_i, \cdot \rangle)|^p \right)^{1/p} \right\|_{C(B_{X^*})}.
$$

\item[(3)] There is  a constant $K>0$ and a Borel probability   measure $\eta$ on $B_{X^*}$  such that for each  $x \in X,$
$$
\|T(x)\| \le K \big( \int_{B_{X^*}}  |\Phi(\langle x, \cdot \rangle)|^p \, d \eta \big)^{1/p}, \quad x \in X.
$$

\item[(4)] There is a  domination space $ (\widehat{X})_{\Phi }^{L^{p}}$ defined by $\Phi$ and a space $L^p(\eta)$ ---for a probability Borel measure  $\eta$ on $B_{X^*}$---, and a continuous linear operator $\widehat T:  (\widehat{X})_{\Phi }^{L^{p}}\to Y $ such that  $T= \widehat T \circ i_{ (\widehat{X})_{\Phi}^{L^{p}}} \circ i$, i.e. the following diagram commutes
$$
\xymatrix{
X \ar[rr]^{T} \ar@{>}[d]_{i} & & Y\\
\widehat X \ar@{>}[rr]^{i_{ (\widehat{X})_{\Phi}^{L^{p}}}} &  & (\widehat{X})_{\Phi}^{L^{p}} \,  \ar@{>}[u]_{\widehat T}} \ \ .
$$

\end{itemize}
\end{theorem}

\begin{proof} We assume without loss of generality that $K=1$.
The implication (1) to (2) is obvious. That (2) implies (3) is a consequence of \cite[Theorem 2.2]{BPR10} or Proposition \ref{1}, taking into account that the dual of $C(B_{X^*})$ is the space of regular Borel measures on $B_{X^*}$. For the implication (3) to (4), note that we can define a domination  space by means of  the seminorm
$$
\| \langle x, \cdot \rangle \|_{ \Phi,L^{p}} := \inf \big\{  \sum_{i=1}^n  \big( \int_{B_{X^*}}  |\Phi(\langle x_i, \cdot \rangle)|^p \,
 d \eta \big)^{1/p}: \, \sum_{i=1}^n x_i=x \big\}, \quad x \in X.
$$
Define $\widehat T([\langle x,\cdot\rangle]):=T(x)$. Let us see that $\widehat T$ is well-defined. If $[\langle x,\cdot\rangle]=[\langle y,\cdot\rangle]$ then $\|\langle x-y,\cdot\rangle \|_{\Phi,L^p}=0$. Given $\epsilon>0$ consider a finite decomposition $x-y=\sum_{i=1}^r z_i$ such that $\sum_{i=1}^r \| j\circ \Phi(\langle z_i,\cdot\rangle )\|_{L^p}<\epsilon.$ By (3),
$$
\|T(x-y)\|\leq \sum_{i=1}^r \|T(z_i)\|\leq \sum_{i=1}^r \Big( \int_{B_{X^*}}|\Phi(\langle z_i,\cdot \rangle )|^p \, d\eta \Big)^{1/p}<\epsilon,
$$
and so $T(x)=T(y)$.

For each $x$ and each finite set $x_1,...,x_n \in X$ such that $\sum_{i=1}^n x_i=x$, we have
$$
\|T(x)\| \le \sum_{i=1}^n \|T(x_i)\| \le \sum_{i=1}^n  \big( \int_{B_{X^*}}  |\Phi(\langle x_i, \cdot \rangle)|^p \, d \eta \big)^{1/p},
$$
and then $\|T(x)\| \le \| \langle x, \cdot \rangle \|_{ \Phi,L^{p}}.$     Hence,
\begin{eqnarray*}
\|\widehat T([ \langle x,\cdot \rangle])\| &=&\inf\{ \| T(y)\|: [ \langle x,\cdot \rangle]=[ \langle y,\cdot \rangle]\}\\
&\leq& \inf \{ \| \langle y,\cdot \rangle\|_{\Phi,L^p}:[ \langle x,\cdot \rangle]=[ \langle y,\cdot \rangle]\}\\
&=&\|[ \langle x,\cdot \rangle]\|_{(\widehat{X})_{\Phi}^{L^{p}}}.\\
\end{eqnarray*}

 By the comments at the beginning of this subsection, if $\Phi$ is bounded then the map $i_{(\widehat X)_\Phi^{L^p}}: \widehat X\to (\widehat X)_\Phi^{L^p}$ is continuous whenever $\widehat X$ is endowed with the supremum norm induced by $C(B_{X^*})$. This gives the factorization in (4).

The proof of the implication (4) to (1) follows from the direct calculation:
\begin{eqnarray*}
\Big(\sum_{i=1}^n \| i_{ (\widehat{X})_{\Phi}^{L^{p}}}(\langle x_i,\cdot\rangle )\|_{ (\widehat{X})_{\Phi}^{L^{p}}}^p\Big)^{1/p} &\leq & \Big(\sum_{i=1}^n \| \langle x_i,\cdot\rangle\|_{ \Phi,L^{p}}^p\Big)^{1/p} \\
&\leq &  \Big(\sum_{i=1}^n \int_{B_{X^*}}|\Phi( \langle x_i,\cdot\rangle)|^p\, d\eta \Big)^{1/p} \\
&\leq & \left\Vert \left( \sum_{i=1}^{n}\left\vert \Phi (\left\langle x_{i},\cdot
\right\rangle )\right\vert ^{p}\right) ^{\frac{1}{p}}\right\Vert
_{C(B_{X^{\ast }})},
\end{eqnarray*}
which shows that $i_{ (\widehat{X})_{\Phi}^{L^{p}}}$ is $\Phi$-abstract $p$-summing.

\end{proof}

Let us finish the section with two particular cases of Theorem \ref{pcoro}.

\begin{itemize}

\item[(1)]
In the case of $p$-summing operators,  we know that the bounded homogeneous map $\Phi$ is the embedding map $\iota: \widehat X \to C(B_{X^*})$. In this case, the domination space is just a closed
subspace of $L_{p}(\mu)$.

\item[(2)]
The domination space for $(p,\sigma )$-absolutely continuous operators,
is the interpolation space $L_{p,\sigma }(\mu )$ (see for example \cite{DAh14}). Recall that  the bounded homogeneous map $\Phi$ is given by $\Phi(\langle x,\cdot\rangle):=\|x\|^\sigma |\langle x,\cdot\rangle|^{1-\sigma}$.

\end{itemize}

\section{Domination spaces for the multilinear case}

\subsection{Strongly $\Phi$-abstract $p$-summing multilinear operators}
Let $X_1, \cdots , X_m$ be Banach spaces.
Consider the compact
space $B_{\mathcal{L}(X_{1},...,X_{m})}$, endowed with the weak* topology
taking into account that $(X_{1}\otimes _{\pi }\cdots \otimes _{\pi
}X_{m})^{\ast }=\mathcal{L}(X_{1},...,X_{m})$. We will denote by $%
i_{m}:X_{1}\times ...\times X_{m}\rightarrow C(B_{(X_{1}\otimes _{\pi
}\cdots \otimes _{\pi }X_{m})^{\ast }})$ the $m$-linear mapping
defined as%
\begin{equation*}
i_{m}(x^{1},...,x^{m})(\varphi ):=\left\langle x^{1}\otimes
...\otimes x^{m},  \varphi \right\rangle ,
\end{equation*}%
for every $x^{j}\in X_{j},$ $j=1,...,m$, and $\varphi \in B_{(X_{1}\otimes
_{\pi }\cdots \otimes _{\pi }X_{m})^{\ast }}.$ Let $V_{m}$ be the linear space
spanned by the set $i_{m}(X_{1}\times ...\times X_{m}).$ The elements of this space are
the (weak*) continuous functions
\begin{equation*}
\sum_{k=1}^{n}\lambda ^{k}\left\langle x^{1,k}\otimes ...\otimes
x^{m,k}, \cdot \right\rangle ,
\end{equation*}%
where $\lambda ^{k}\in \mathbb{R}$ and $x^{j,k}\in X_{j},$ $j=1,...,m,$ $%
k=1,...,n$.

Let $\Phi :V_{m}\rightarrow C(B_{(X_{1}\otimes _{\pi }\cdots \otimes _{\pi
}X_{m})^{\ast }})$ be a  homogeneous mapping and
let $1\leq p <\infty $.

\begin{definition}
A multilinear operator $T:X_{1}\times \cdots \times X_{m}\rightarrow Y$ is
{\it strongly $\Phi$-abstract $p$-summing} if there is a constant $C>0$ such that for
each $x_{i}^{j,k}\in X_{j}$, $(j=1,...,m)$, and scalars $\lambda _{i}^{k},$ $%
i=1,...,n_{1},$ $k=1,...,n_{2}$ and all $n_{1},n_{2}\in \mathbb{N},$ we have%
\begin{equation*}
\left( \sum_{i=1}^{n_{1}}\left\Vert \sum_{k=1}^{n_{2}}\lambda
_{i}^{k}T(x_{i}^{1,k},...,x_{i}^{m,k})\right\Vert ^{p}\right) ^{1/p}
\end{equation*}%
\begin{equation}
\leq C\left\Vert \left( \sum_{i=1}^{n_{1}}\left\vert \Phi \left(
\sum_{k=1}^{n_{2}}\lambda _{i}^{k}\left\langle x_{i}^{1,k}\otimes
...\otimes x_{i}^{m,k}, \cdot \right\rangle \right) \right\vert ^{p}\right) ^{\frac{%
1 }{p}}\right\Vert _{C(B_{(X_{1}\otimes _{\pi }\cdots \otimes _{\pi
}X_{m})^{\ast }})}.  \label{def2}
\end{equation}
\end{definition}

Note that the class of factorable strongly $p$-summing multilinear
operators introduced in \cite{PeRuSa} is an
example of this definition just defining $\Phi $ as the natural embedding map
$\iota: V_{m}\rightarrow C(B_{(X_{1}\otimes _{\pi }\cdots \otimes _{\pi
}X_{m})^{\ast }})$. The strongly $p$-summing multilinear operators of Dimant
(see \cite{Dim03}) are also related with our class, but in this case  only $n_{2}=1$ is allowed in the  inequalities considered in the definition.

The next result gives the characterization of
strongly $\Phi$-abstract $p$-summing $m$-linear mappings by an integral domination.

\begin{theorem}
\label{domin} Let $T:X_{1}\times ...\times X_{m}\rightarrow Y$ be a $m$%
-linear mapping among Banach spaces and let $\Phi $ be a  homogeneous
map $\Phi :V_{m}\rightarrow C(B_{(X_{1}\otimes _{\pi }\cdots \otimes _{\pi
}X_{m})^{\ast }})$. The mapping $T$ is strongly
$\Phi$-abstract $p$-summing\ if and only if there is a regular Borel probability
measure $\eta $ on $B_{(X_{1}\otimes _{\pi }\cdots \otimes _{\pi
}X_{m})^{\ast }}$ and a constant $C>0$ such that for each $x^{j,k}\in X,$ $%
\lambda ^{k}\in \mathbb{R}$ with $j=1,...,m$ and $k=1,...,n$ we have%
\begin{equation*}
\left\Vert \sum_{k=1}^{n}\lambda ^{k}T(x^{1,k},...,x^{m,k})\right\Vert
\end{equation*}%
\begin{equation*}
\leq C\left( \int_{B_{(X_{1}\otimes _{\pi }\cdots \otimes _{\pi
}X_{m})^{\ast }}}\left\vert \Phi \left( \sum_{k=1}^{n}\lambda
^{k}\left\langle x^{1,k}\otimes ...\otimes x^{m,k}, \cdot \right\rangle
\right) \right\vert ^{p}\,d\eta \right) ^{\frac{1}{p}}.
\end{equation*}
\end{theorem}

\begin{proof} We write $Z^{(\mathbb N)}$ for all sequences in a Banach space $Z$ that are eventually null.
The $m$-linear mapping $T$ is $R$-$S$-abstract $p$-summing  for
$$
R:B_{(X_{1}\otimes _{\pi
}\cdots \otimes _{\pi }X_{m})^{\ast }}\times (\mathbb R\times X_1,\ldots, X_m)^{(\mathbb N)} \times \mathbb R\to [0,\infty)
$$
and
$$
S:{\mathcal L}(X_{1},\ldots,X_{m};Y)\times (\mathbb R\times X_1,\ldots, X_m)^{(\mathbb N)} \times \mathbb R\to [0,\infty),
$$
given by
\begin{equation*}
R\Big(\phi,\big( ( \lambda^1, x^{1,1},\ldots,x^{m,1}),\ldots, (\lambda^n,x^{1,n},\ldots,x^{m,n})\big), b\Big):=\Big| \Phi (\sum_{k=1}^n \lambda^k \langle x^{1,k} \otimes \cdots \otimes x^{m,n},\cdot\rangle )(\phi)\Big|
\end{equation*}%
and%
\begin{equation*}
S\Big(f,\big( ( \lambda^1, x^{1,1},\ldots,x^{m,1}),\ldots, (\lambda^n,x^{1,n},\ldots,x^{m,n})\big), b\Big):=\Big\| \sum_{k=1}^n \lambda^k f( x^{1,k} , \ldots , x^{m,n} )\Big\| .
\end{equation*}%
Then  \cite{BPR10} or \cite{ PeSa} gives the result.

\end{proof}

Note that Theorem \ref{domin} can be proved using \cite[Theorem 2.2]{BPR10} (see also the proof of Theorem \ref{dom2} below.)

Now we give the main result of this section, that generalizes Theorem 3.3 in
\cite{PeRuSa}.

\begin{theorem} \label{facpara}
\label{facto} Let $1 \le p < \infty$. Let $T:X_{1}\times ...\times X_{m}\rightarrow Y$ be an $m$%
-linear mapping between Banach spaces and let $\Phi $ a bounded homogeneous map  $\Phi :V_{m}\rightarrow
C(B_{(X_{1}\otimes _{\pi }\cdots \otimes _{\pi }X_{m})^{\ast }})$.
Then the mapping $T$ is strongly $\Phi$-abstract $p$-summing\ if and only if there
is a domination space $(\widehat{V_m})_{\Phi }^{L^p}$ defined by $\Phi $ and a $L^p$-space, and a continuous linear operator $\widehat T: (\widehat{V_m})_{\Phi }^{L^p} \to Y$ such that  $T=\widehat{T}\circ
i_{(\widehat{V_m})_{\Phi }^{L^p}}\circ i_m$, i.e.
\begin{equation*}
\xymatrix{ X_{1}\times ...\times X_{m} \ar[rr]^{T} \ar@{>}[d]_{i_{m}} & & Y .\\
V_{m} \ar@{>}[rr]^{i_{(\widehat{V_m})_{\Phi }^{L^p}}} & & (\widehat{V_m})_{\Phi }^{L^p} \,  \ar@{>}[u]_{\widehat T}}
\end{equation*}
\end{theorem}

\begin{proof}
By Theorem \ref{domin}, if $T$ is strongly $\Phi$-abstract $p$-summing there is
a regular probability measure $\eta $ on $B_{(X_{1}\otimes _{\pi }\cdots
\otimes _{\pi }X_{m})^{\ast }}$ and a constant $C>0$ such that for every $x^{j,k}\in X,$ $%
\lambda ^{k}\in \mathbb{R}$ with $j=1,...,m$ and $k=1,...,n$%
\begin{equation*}
\left\Vert \sum_{k=1}^{n}\lambda ^{k}T(x^{1,k},...,x^{m,k})\right\Vert
\end{equation*}%
\begin{equation*}
\leq C \left( \int_{B_{(X_{1}\otimes _{\pi }\cdots \otimes _{\pi }X_{m})^{\ast
}}}\left\vert \Phi \left( \sum_{k=1}^{n}\lambda ^{k}\left\langle x^{1,k}\otimes ...\otimes x^{m,k}, \cdot \right\rangle \right) \right\vert
^{p}\,d\eta \right) ^{\frac{1}{p}}.
\end{equation*}%
Therefore, in this case we consider the Banach function space $Z=L^p(\eta)$ and its related  domination space $(\widehat{V_m})_{\Phi }^{L^p}$, whose norm is given by

\begin{equation*}
\left\Vert \Big[\sum_{k=1}^{n}\lambda ^{k}\left\langle x^{1,k}\otimes
...\otimes x^{m,k}, \cdot \right\rangle \Big] \right\Vert _{(\widehat{V_m})_{\Phi }^{L^p} }
\end{equation*}%
\begin{equation*}
:=\inf \sum_{i=1}^{r}\left( \int_{B_{(X_{1}\otimes _{\pi }\cdots \otimes
_{\pi }X_{m})^{\ast }}}\left\vert \Phi \left( \sum_{k=1}^{n}\lambda
_{i}^{k}\left\langle x_{i}^{1,k}\otimes ...\otimes
x_{i}^{m,k}, \cdot \right\rangle \right) \right\vert ^{p}\,d\eta \right) ^{\frac{1}{p%
}},
\end{equation*}%
where the infimum is taken over all representations of $\sum_{k=1}^{n}%
\lambda ^{k}x^{1,k}\otimes ...\otimes x^{m,k}$ of the form $%
\sum_{i=1}^{r}\left( \sum_{k=1}^{n}\lambda _{i}^{k}x_{i}^{1,k}\otimes
...\otimes x_{i}^{m,k}\right) $.
\end{proof}

The following corollary gives the connection between a strongly $\Phi$-abstract $p$-summing multilinear operators and its linearization.

\begin{corollary} Let $\Phi: V_m \to C(B_{(X_1 \otimes_\pi \cdots \otimes_\pi X_m)^*})$ be homogeneous  and let $1 \le p < \infty$.
A multilinear mapping $T:X_{1}\times ...\times X_{m}\rightarrow Y$ is strongly $\Phi$-abstract $p$-summing if and only if
its linearization $T_{L}$ is a $\Phi$-abstract $p$-summing linear operator.
\end{corollary}

\begin{proof}
Note that, by the canonical factorization of the multilinear maps through
the projective tensor product, we have that the factorization in Theorem \ref{facpara} can be written as%
\begin{equation*}
\xymatrix{ &X_{1}\times ...\times X_{m}\ar[dl]^{\sigma_m} \ar[rr]^{T} \ar@{>}[d]_{i_m} & & Y\\
X_{1} {\otimes }_{\pi }\cdots  {\otimes }_{\pi }X_{m}\ar[r]^{I_m}& V_m \ar@{>}[rr]^{i_{(\widehat{V_m})_{\Phi }^{L^p}}} & &(\widehat{V_m})_{\Phi }^{L^p} \
\ar@{>}[u]_{\widehat T}}
\end{equation*}%
where $I_{m}$ is the the isometric embedding $X_{1} {\otimes }_{\pi
}\cdots  {\otimes }_{\pi }X_{m}\longrightarrow V_{m}$ given by
\begin{equation*}
I_{m}(x^{1}\otimes \cdots \otimes x^{m})=\left\langle x^{1}\otimes
\cdots \otimes x^{m}, \cdot \right\rangle
\end{equation*}%
It is clear that $\widehat{T}\circ i_{(\widehat{V_m})_{\Phi }^{L^p}}\circ I_{m}=T_{L}$ and so by Theorem \ref{pcoro} $T
$ is strongly $\Phi$-abstract $p$-summing if and only if $T_{L}$ is $\Phi$-abstract $p$-summing. Note that, using the notation introduced in Section 2, we have that  $V_m= \widehat{(X_{1} {\otimes }_{\pi }\cdots {\otimes }_{\pi }X_{m}})$.
\end{proof}

\subsection{$(\Phi _{1},...,\Phi _{m})$-abstract $(p_{1},...,p_{m})$-summing multilinear operators}$\,$
\newline

Throughout this section, $1\leq p,p_{1},...,p_{m}<\infty $ are such that $\frac{1}{p}=\frac{1}{p_{1}}%
+...+\frac{1}{p_{m}}$ and  $X_{1},...,X_{m},Y$ are Banach spaces. Let also $%
\Phi _{j}:\widehat{X}_{j}\longrightarrow C(B_{X_{j}^{\ast }}),$ $%
j=1,...,m,$ be bounded  homogeneous maps.

\begin{definition}
An $m$-linear operator $T\in \mathcal{L}(X_{1},...,X_{m};Y)$ is said to be $(\Phi _{1},...,\Phi _{m})$-abstract $(p_{1},...,p_{m})$-summing if there is a constant $%
C>0$ so that%
\begin{equation*}
\left( \overset{n}{\underset{i=1}{\sum }}\left\Vert
T(x_{i}^{1},...,x_{i}^{m})\right\Vert ^{p}\right) ^{\frac{1}{p}}\leq C%
\overset{m}{\underset{j=1}{\prod }}\left\Vert \left( \overset{n}{\underset{%
i=1}{\sum }}  \left\vert  \Phi _{j} \left( \left\langle x_{i}^{j},.\right\rangle
 \right) \right\vert^{p_{j}}\right) ^{\frac{1}{p_{j}}}\right\Vert
_{C(B_{X_{j}^{\ast }})}.
\end{equation*}
\end{definition}

Using the full general Pietsch Domination Theorem presented by
Pellegrino et all in \cite{PeSaSe} we obtain the Domination Theorem for our
class.

\begin{theorem} \label{dom2}
An  operator $T\in \mathcal{L}(X_{1},...,X_{m};Y)$ is
$(\Phi _{1},...,\Phi _{m})$-abstract $(p_{1},...,p_{m})$-summing if and only if there is a
constant $C>0$ and regular Borel probability measures $\mu _{j}\in
C(B_{X_{j}^{\ast }})^{\ast }$ such that for all $x^{j}\in X_{j}(j=1,...,m),$
\begin{equation}
\left\Vert T(x^{1},...,x^{m})\right\Vert \leq C\overset{m}{\underset{j=1}{%
\prod }}\left( \int_{B_{X_{j}^{\ast }}} \left\vert \Phi _{j}\left(
\left\langle x^{j},.\right\rangle \right)  \right\vert^{p_{j}}d\mu
_{j}\right) ^{\frac{1}{p_{j}}}.  \label{dom9}
\end{equation}
\end{theorem}

\begin{proof}
The $m$-linear mapping $T$ is $R_{1},...,R_{m}$-$S$-abstract $(p_{1},...,p_{m})
$-summing (see \cite[Definition 4.4]{PeSaSe}), for
\begin{equation*}
R_{j}(\varphi ,b^{1},...,b^{r},x^{j})=  \left\vert \Phi _{j}\left(
\left\langle x^{j},\varphi \right\rangle  \right)\right\vert
\end{equation*}%
and%
\begin{equation*}
S(T,b^{1},...,b^{r},x^{1},...,x^{m})=\left\Vert T(x^{1},...,x^{m})\right\Vert
\end{equation*}%
for all $\varphi \in B_{X_{j}^{\ast }},b^{l}\in \mathbb{K},x^{j}\in X_{j}$
with $j=1,...,m.$ Then  Theorem 4.6 in \cite{PeSaSe} gives the result.
\end{proof}
Using the factorization for the linear case (Theorem \ref{pcoro}), we are now ready to construct
the domination space that provides  the factorization theorem for
our class.

\begin{theorem} \label{facfi}
The $m$-linear operator $T\in \mathcal{L}(X_{1},...,X_{m};Y)$ is $(\Phi _{1},...,\Phi _{m})$-abstract $(p_{1},...,p_{m})$-summing if and only if there are Banach spaces $%
G_{j}$ and  $\Phi _{j}$-abstract $p_{j}$-summing linear operators $u_{j}:X_{j}\longrightarrow G_{j}$, j=1,...m,  such that $T$ factors through $G_{1}\times ...\times G_{m}$
as $T=\widehat{T}\circ (u_{1},...,u_{m}),$ where $\widehat{T}\in \mathcal{L}%
(G_{1},...,G_{m};Y)$.
\end{theorem}

\begin{proof}
A simple computation using  Theorem \ref{pcoro} and Theorem \ref{dom2} shows that,  if $T$
has such a factorization then $T$ is $(\Phi _{1},...,\Phi _{m})$-abstract $(p_{1},...,p_{m})$-summing.

Conversely, take $T$ in our class. For each $j=1,...m$, consider the
 map $u_{j}:X_{j} \to  (\widehat{X_j})_{\Phi_j }^{L^{p_j}}$ defined  by $u_j:= i_{ (\widehat{X_j})_{\Phi_j }^{L^{p_j}}} \circ i_j$, where   $i_j:X_j \to \widehat{X}_j$  is the identification map and $ i_{ (\widehat{X_j})_{\Phi_j }^{L^{p_j}}}:\widehat{X}_j \to   (\widehat{X_j})_{\Phi_j }^{L^{p_j}}$ is the quotient map defined as
was explained in Section 2. Notice that we have
\begin{equation*}
\begin{array}{lll}
\left\Vert u_{j}(x^{j})\right\Vert  & = & \inf \left\{ \overset{n}{\underset{%
k=1}{\sum }}\left( \int_{B_{X_{j}^{\ast }}}   \left\vert \Phi _{j}\left(
\left\langle x_{k}^{j},.\right\rangle  \right) \right\vert^{p_{j}}d\mu
_{j}\right) ^{\frac{1}{p_{j}}},x^{j}=\overset{n}{\underset{k=1}{\sum }}%
x_{k}^{j}\right\}  \\
& \leq  & \left( \int_{B_{X_{j}^{\ast }}}\left\vert \Phi _{j}\left(
\left\langle x^{j},.\right\rangle \right)  \right\vert^{p_{j}}d\mu
_{j}\right) ^{\frac{1}{p_{j}}}, %
\end{array}%
\end{equation*}%
and so $u_{j}$ is $\Phi _{j}$-abstract $p_{j}$-summing for all $j=1,...m.$ Let $\widehat{T}%
_{0}$ be the $m$-linear mapping defined on $u_{1}(X_{1})\times ...\times
u_{m}(X_{m})$ by
\begin{equation*}
\widehat{T}_{0}(u_{1}(x^{1})\times ...\times u_{m}(x^{m})):=T(x^{1},...,x^{m}),
\end{equation*}%
that is continuous due to the inequality  (\ref{dom9}) that provides Theorem \ref{dom2}.
(Following the idea of \cite[Theorem 3.6 ]{DAS12} we can easily prove that the mapping $%
\widehat{T}_{0}$ is well-defined and continuous). We finish the
proof by defining $\widehat{T}$ as the unique extension of $\widehat{T}_{0}$ to $%
\overline{u_{1}(X_{1})}\times ...\times \overline{u_{m}(X_{m})}=
 (\widehat{X_1})_{\Phi_1 }^{L^{p_1}}\times ...\times  (\widehat{X_m})_{\Phi_m}^{L^{p_m}}.$
\end{proof}

\begin{example}
The situation considered in the preceding theorem includes relevant well-known cases. The next classes of operators satisfy  factorization theorems  associated to
  summability properties that are particular cases of  Theorem \ref{facfi}.
\begin{itemize}

\item[(1)]
\textbf{$(p_{1},...,p_{m})$-dominated multilinear operators} (see \cite{14} or \cite{10}). Let $1\leq p,p_{1},...,p_{m}<\infty $ with $\frac{1}{p}=\frac{1}{p_{1}}+...+%
\frac{1}{p_{m}}$. A multilinear operator $T\in \mathcal{L}(X_{1},...,X_{m};Y)
$ is {\it $(p_{1},...,p_{m})$-dominated} if there is a constant $C>0$ such that
\begin{equation*}
\left( \overset{n}{\underset{i=1}{\sum }}\left\Vert
T(x_{i}^{1},...,x_{i}^{m}\right\Vert ^{p}\right) ^{\frac{1}{p}}\leq C\overset%
{m}{\underset{j=1}{\prod }}\underset{x_{j}^{\ast }\in B_{X_{j}^{\ast }}}{%
\sup }\left( \overset{n}{\underset{i=1}{\sum }}\left\vert \left\langle
x_{i}^{j},x_{j}^{\ast }\right\rangle \right\vert ^{p_{j}}\right) ^{\frac{1}{%
p_{j}}}
\end{equation*}
for all $(x_{i}^{j})_{1\leq i\leq n}\subset X_{j}$, $j=1,...,m$. In this case, the corresponding homogeneous maps $\Phi _{j}$ are the
inclusion maps $\widehat{X}_{j}\hookrightarrow C(B_{X_{j}^{\ast }})$ and the
corresponding seminorms are
\begin{equation*}
\left\Vert \left\langle x^{j},.\right\rangle \right\Vert _{ (\widehat{X_j})_{\Phi_j }^{L^{p_j}}}=\inf \left\{ \overset{n}{\underset{k=1}{\sum }}\left(
\int_{B_{X_{j}^{\ast }}}\left\vert \left\langle x_{k}^{j},.\right\rangle
\right\vert ^{p_{j}}d\eta _{j}\right) ^{\frac{1}{p_{j}}},x^{j}=\overset{n}{%
\underset{k=1}{\sum }}x_{k}^{j}\right\}.
\end{equation*}%
The domination  spaces $ (\widehat{X_j})_{\Phi_j }^{L^{p_j}}$ coincide
with a subspace of $L^{p_{j}}(\eta _{j})$,  $j=1,...,m$, that provide the factorization theorem for this class (see Theorem 3.2.4 in \cite{Gei84} and Theorem 14 in \cite{14}).

\item[(2)]
\textbf{Dominated $(p_{1},...,p_{m};\sigma )$-continuous multilinear
operators} (see \cite{DAS12} or \cite{DAh14}). Let $1\leq
p,p_{1},...,p_{m}<\infty $ and $0\leq \sigma <1$ with $\frac{1}{p}=\frac{1}{%
p_{1}}+...+\frac{1}{p_{m}}$. A multilinear operator $T\in \mathcal{L}%
(X_{1},...,X_{m};Y)$ is {\it dominated $(p_{1},...,p_{m};\sigma )$-continuous}  if
there is a constant $C>0$ such that
\begin{equation*}
\left( \overset{n}{\underset{i=1}{\sum }}\left\Vert
T(x_{i}^{1},...,x_{i}^{m})\right\Vert ^{\frac{p}{1-\sigma }}\right) ^{\frac{%
1-\sigma }{p}}\leq C\overset{m}{\underset{j=1}{\prod }}\underset{x_{j}^{\ast
}\in B_{X_{j}^{\ast }}}{\sup }\left( \overset{n}{\underset{i=1}{\sum }}%
\left( \left\vert \left\langle x_{i}^{j},x_{j}^{\ast }\right\rangle
\right\vert ^{1-\sigma }\left\Vert x_{i}^{j}\right\Vert ^{\sigma }\right) ^{%
\frac{p_{j}}{1-\sigma }}\right) ^{\frac{1-\sigma }{p_{j}}}
\end{equation*}%
for all $(x_{i}^{j})_{1\leq i\leq n}\subset X_{j},$ $j=1,...,m$. The
corresponding homogeneous maps $\Phi _{j}:\widehat{X}_{j}\longrightarrow
C(B_{X_{j}^{\ast }})$ are defined by%
\begin{equation*}
\Phi _{j}( \left\langle x^{j},.\right\rangle
)=\left\vert \left\langle x^{j},.\right\rangle \right\vert ^{1-\sigma
}\left\Vert x^{j}\right\Vert ^{\sigma },
\end{equation*}
the corresponding seminorms are given by%
\begin{equation*}
\left\Vert \left\langle x^{j},.\right\rangle \right\Vert _{ (\widehat{X_j})_{\Phi_j }^{L^{p_j}}}=\inf \left\{ \overset{n}{\underset{k=1}{\sum }}\left\Vert
x_{k}^{j}\right\Vert ^{\sigma }\left( \int_{B_{X_{j}^{\ast }}}\vert
\langle x_{k}^{j},.\rangle \vert ^{p_{j}}d\eta _{j}\right)
^{\frac{1-\sigma }{p_{j}}},x^{j}=\overset{n}{\underset{k=1}{\sum }}%
x_{k}^{j}\right\}
\end{equation*}%
and the domination spaces $ (\widehat{X_j})_{\Phi_j }^{L^{p_j}}$ coincide with the spaces $%
L_{p_{j},\sigma }(\eta _{j}),$ $j=1,...,m$, that were described in \cite[Section 2.1.3]
{DAh14}. Then  Theorem \ref{facfi} when applied with these elements gives the  factorization result that can be found in \cite[Theorem 2.1.20]{DAh14} (see also
 \cite{adsp14} and \cite{DAS12}).
\end{itemize}

\end{example}

\section{Applications: the particular case of factorable strongly summing multilinear operators}

Dimant introduced in  \cite{Dim03} the ideal  of strongly $p$%
-summing multilinear operators as follows.

\noindent A continuous\ $m$-linear operator $T:X_{1}\times...\times X_{m}\longrightarrow Y$
is {\it  strongly $p$-summing} if there exists a constant $C>0$ such that for every $%
x_{1}^{j},...,x_{n}^{j}\in X_{j},$  $j=1,...,m,$ we have
\begin{equation}
\left\Vert \left( T(x_{i}^{1},...,x_{i}^{m})\right) _{i=1}^{n}\right\Vert
_{p}\leq C\underset{\varphi \in B_{\mathcal{L}\left( X_{1},...,X_{m}\right) }}{%
\sup }\left( \overset{n}{\underset{i=1}{\sum }}\left\vert \varphi
(x_{i}^{1},...,x_{i}^{m})\right\vert ^{p}\right) ^{\frac{1}{p}}.
\label{Difdim03}
\end{equation}
In this section we will follow the notation of the original presentation of Dimant according to the above definition.
\noindent We will denote by $\mathcal{L}_{p}^{s}\left(
X_{1},...,X_{m};Y\right) $ the vector space of all strongly $p$-summing $m$%
-linear operators $T$ from $X_{1}\times ...\times X_{m}$ into $Y$, which is
a Banach space if we consider the norm $\left\Vert T\right\Vert _{\mathcal{L}
_{p}^{s}},$ the infimum of all $C$ satisfying (\ref{Difdim03}). Even if this class satisfies a domination theorem, it is well-known that it does not satisfy a factorization theorem. So, no factorization scheme is expected whenever we try to extend this class via an interpolation procedure. Matter  initiated this general interpolation technique in  \cite{Matt87} where he introduces the class of $(p,\sigma)$-absolutely continuous linear operators, that recovers the $p$-summing operators for $\sigma=0$. The natural interpolation class related to strongly $p$-summing $m$-linear maps can be constructed as follows.

Let $1\leq p<\infty $ and $0\leq \sigma <1$. For all $(x_{i}^{j})_{i=1}^{n}%
\subset X_{j},(1\leq j\leq m)$ we put
\begin{equation*}
\delta _{p\sigma }((x_{i}^{j})_{i=1}^{n})=\underset{\varphi \in B_{\mathcal{L}%
\left( X_{1},...,X_{m}\right) }}{\sup }\left( \underset{i=1}{\overset{n}{%
\sum }}\left( \left\vert \varphi (x_{i}^{1},...,x_{i}^{m})\right\vert
^{1-\sigma }\overset{m}{\underset{j=1}{\prod }}\left\Vert
x_{i}^{j}\right\Vert ^{\sigma }\right) ^{\frac{p}{1-\sigma }}\right) ^{\frac{%
1-\sigma }{p}}.
\end{equation*}

\noindent It is clear that
\begin{equation*}
\underset{\varphi \in B_{\mathcal{L}\left( X_{1},...,X_{m}\right) }}{\sup }%
\left( \overset{n}{\underset{i=1}{\sum }}\left\vert \varphi
(x_{i}^{1},...,x_{i}^{m})\right\vert ^{\frac{p}{1-\sigma }}\right) ^{\frac{%
1-\sigma }{p}}\leq \delta _{p\sigma }((x_{i}^{j})_{i=1}^{n}),
\end{equation*}
\noindent for all $(x_{i}^{j})_{i=1}^{n}\subset X_{j},$ $1\leq j\leq m.$

\begin{definition}
For $1\leq p<\infty $ and $0\leq \sigma $ $<1,$ a mapping $T\in \mathcal{L}%
(X_{1},...,X_{m};Y)$ is \textit{ Dimant strongly $(p,\sigma)$-continuous} if there is a
constant $C>0$ such that for any $x_{1}^{j},...,x_{n}^{j}\in X_{j},$ $1\leq
j\leq m$, we have

\begin{equation}
\left\Vert \left( T\left( x_{i}^{1},...,x_{i}^{m}\right) \right)
_{i=1}^{n}\right\Vert _{\frac{p}{1-\sigma }}\leq C \ \delta _{p\sigma
}((x_{i}^{j})_{i=1}^{n}).  \label{Ourdif}
\end{equation}
\end{definition}

\noindent The class of all Dimant strongly $(p,\sigma )$-continuous $m$%
-linear operators from $X_{1}\times ...\times X_{m}$ into $Y$, which is
denoted by $\mathcal{L}_{p}^{s,\sigma }\left( X_{1},...,X_{m};Y\right) $ is
a Banach space with the norm $\left\Vert T\right\Vert _{\mathcal{L}%
_{p}^{s,\sigma }}$ which is the smallest constant $C$ such that the
inequality (\ref{Ourdif}) holds.

\noindent For $\sigma =0,$ we have $\mathcal{L}_{p}^{s,0}\left(
X_{1},...,X_{m};Y\right) =\mathcal{L}_{p}^{s}\left( X_{1},...,X_{m};Y\right)
.$ Actually this vector space has to be thought of as an intermediate space between
the space of all strongly $p$-summing mappings ($\sigma =0$) and the whole
class of continuous $m$-linear mappings ($\sigma =1$).

\begin{theorem}
\label{dom23} Let $1\leq p\leq \infty $ and $0\leq \sigma <1.$ A $m$-linear
operator $T\in \mathcal{L} \left( X_{1},...,X_{m};Y\right)$ is Dimant strongly $%
(p,\sigma)$-continuous if and only if there is a constant $C>0$ and regular
Borel probability measure $\mu $ on $B_{\mathcal{L}\left(
X_{1},...,X_{m}\right)}$ so that for all $\left( x^{1},...,x^{m}\right) \in
X_{1}\times ....\times X_{m}$ the inequality

\begin{equation}
\left\Vert T(x^{1},...,x^{m})\right\Vert \leq C\left( \int_{B_{\mathcal{L}%
\left( X_{1},...,X_{m}\right) }}(\left\vert \phi
(x^{1},...,x^{m})\right\vert ^{1-\sigma }\overset{m}{\underset{j=1}{\prod }}%
\left\Vert x^{j}\right\Vert ^{\sigma })^{\frac{p}{1-\sigma }}d\mu (\phi
)\right) ^{\frac{1-\sigma }{p}},  \label{Ourdomin}
\end{equation}

is valid. Moreover, we have in this case
\begin{equation*}
\left\Vert T\right\Vert _{\mathcal{L}_{p}^{s,\sigma}}=\inf \left\{ C>0:\text{%
\textit{for all }}C\text{ satisfying the \textit{inequality} (\ref{Ourdomin})}
\right\}
\end{equation*}
\end{theorem}

\begin{proof}
For the proof we use the abstract version of Pietsch's domination theorem
presented by Botelho et al in \cite{BPR10}. Note that by choosing the
parameters

$\left\{
\begin{array}{l}
X=E=X_{1}\times ...\times X_{m} \\
\mathcal{H}=\mathcal{L}\left( ^{m}X_{1},...,X_{m};Y\right)  \\
x_{o}=(0,...,0) \\
K=B_{\mathcal{L}\left( X_{1},...,X_{m}\right) } \\
G=\mathbb{K} \\
R(\varphi ,x^{1},...,x^{m},\lambda )=\left\vert \lambda \right\vert
\left\vert \varphi (x^{1},...,x^{m})\right\vert ^{1-\sigma }\overset{m}{%
\underset{j=1}{\prod }}\left\Vert x_{i}^{j}\right\Vert ^{\sigma } \\
S(T,(x^{1},...,x^{m}),\lambda )=\left\vert \lambda \right\vert \left\Vert
T(x^{1},...,x^{m})\right\Vert.
\end{array}%
\right. $

\noindent We can easily conclude that $T:X_{1}\times ...\times
X_{m}\longrightarrow Y$ is Dimant strongly $(p,\sigma )$-continuous if
and only if $T$ is $R$-$S$- abstract $\frac{p}{1-\sigma }$-summing (see \cite%
[Definition 2.1]{BPR10}), and by \cite[Theorem 2.2]{BPR10} we recover (\ref%
{Ourdomin}).
\end{proof}

An immediate consequence of Theorem \ref{dom23} is the following corollary.

\begin{corollary}
Consider $0\leq \sigma <1$ and $1\leq p,q<\infty$ with  $p\leq q.$
If $T\in \mathcal{L}_{p}^{s,\sigma }\left( X_{1},...,X_{m};Y\right) $
\textit{then} $T\in \mathcal{L}_{q}^{s,\sigma }\left(
X_{1},...,X_{m};Y\right) $\textit{\ and }$\left\Vert T\right\Vert _{\mathcal{%
L}_{q}^{s,\sigma }}\leq \left\Vert T\right\Vert _{\mathcal{L}_{p}^{s,\sigma
}}$
\end{corollary}

\begin{proposition}
Let $1\leq p<\infty ,0\leq \sigma <1$. Then
\begin{equation*}
\mathcal{L}_{\frac{p}{1-\sigma }}^{s}\left( X_{1},...,X_{m};Y\right) \subset
\mathcal{L}_{p}^{s,\sigma }\left( X_{1},...,X_{m};Y\right).
\end{equation*}
\end{proposition}

By \cite[Example 3.3]{CD03} there is a strongly $p$-summing bilinear
operator which is not weakly compact. Then the previous proposition proves
the existence of a Dimant strongly $(p,\sigma )$-continuous bilinear operator which
is not weakly compact.

The lack of factorization scheme for the class of Dimant strongly $(p,\sigma)$-continuous operators highlights the need of approaching the definition to a factorable context. A precedent was done first in \cite{RuSa}, where a general condition related to the existence of a factorization scheme was isolated.  Combining the ideas in \cite{RuSa} with those in \cite{Dim03}, the class of factorable strongly $p$-summing operators was introduced in \cite{PeRuSa}, that best inherits the spirit of absolutely summing linear operators. Following this scheme, we now combine Dimant strongly $(p,\sigma)$-continuous operators with the ideas in \cite{RuSa} to produce a class for which a factorization theorem occurs.

\begin{definition} \label{demul}
For $1\leq p<\infty $ and $0\leq \sigma $ $<1,$ a mapping $T\in \mathcal{L}%
(X_{1},...,X_{m};Y)$ is factorable strongly $(p,\sigma )$-continuous if
there is a constant $C>0$ such that for any natural numbers $n_{1},n_{2}$\
and vectors $x_{i,k}^{j}\in X_{j},$ $1\leq j\leq m,$ and scalars $\lambda
_{i}^{k},$ $i=1,...,n_{1},$ $k=1,...,n_{2}$ we have
\begin{equation*}
\left( \sum_{i=1}^{n_{1}}\left\Vert \sum_{k=1}^{n_{2}}\lambda
_{i}^{k}T(x_{i}^{1,k},...,x_{i}^{m,k}) \right\Vert ^{\frac{p}{1-\sigma }%
}\right) ^{\frac{1-\sigma }{p}}
\end{equation*}%
\begin{equation*}
\leq C.\sup_{\varphi \in B_{\mathcal{L}(X_{1}\times \cdots \times
X_{n})}}\left( \sum_{i=1}^{n_{1}}\left\vert \sum_{k=1}^{n_{2}} |\lambda
_{i}^{k}| |\varphi (x_{i}^{1,k},...,x_{i}^{m,k})|^{1-\sigma }\underset{j=1}{\overset%
{m}{\prod }}\Vert x_{i}^{j,k}\Vert ^{\sigma }\right\vert ^{\frac{p}{1-\sigma
}}\right) ^{\frac{1-\sigma }{p}}.
\end{equation*}
\end{definition}

\begin{remark}
1) It is clear that this definition is equivalent to saying that $T$ is
strongly $\Phi$-abstract $\frac{p}{1-\sigma }$-summing with the homogeneous function
$\Phi $  defined ---using our notation--- by%
\begin{equation*}
\Phi \left( \sum_{k=1}^{n}\lambda _{i}^{k}\left\langle x^{1,k}\otimes
...\otimes x^{m,k}, \cdot \right\rangle \right) = \inf \left\{ \sum_{k=1}^{n} |\lambda
_{i}^{k} | \left|\left\langle x^{1,k}\otimes ...\otimes x^{m,k}, \cdot \right\rangle
\right|^{1-\sigma }\underset{j=1}{\overset{m}{\prod }}\Vert x^{j,k}\Vert ^{\sigma } \right\},
\end{equation*}%
where the infimum is computed for all possible decompositions of the tensor
$$
\sum_{k=1}^{n}\lambda _{i}^{k} x^{1,k}\otimes
...\otimes x^{m,k}.
$$

2) Note that if we take $n_{2}=1$ and $\lambda _{i}^{1}=1$\ for all $%
i=1,...,n_{1}$ in the Definition \ref{demul} we find that $T$ is Dimant strongly $%
(p,\sigma )$-continuous $m$-linear mappings.
\end{remark}

From  Theorem \ref{domin} and Theorem \ref{facto} we obtain the following
characterization.

\begin{theorem}
Let $1\leq p<\infty $, $0\leq \sigma $ $<1$ and $T\in \mathcal{L%
}(X_{1},...,X_{m};Y)$. The following assertions are equivalent.

\begin{itemize}
\item[(1)] $T$ is factorable strongly $(p,\sigma )$-continuous.

\item[(2)] There is a constant $C>0$ and a regular Borel probability measure $\eta $ on $%
B_{(X_{1}\widehat{\otimes }_{\pi }\cdots \widehat{\otimes }_{\pi }X_{m})^{\ast }}$ such that for
all $x^{j,k}\in X_{j}$ and $\lambda ^{k}\in \mathbb{R}$, $j=1,...,m$, $%
k=1,...,n,$%
\begin{equation*}
\left\Vert \sum_{k=1}^{n}\lambda ^{k}T(x^{1,k},...,x^{m,k})\right\Vert
\end{equation*}%
\begin{equation*}
\leq C \left( \int_{B_{\mathcal{L}%
\left( X_{1},...,X_{m}\right) }}\left\vert \sum_{k=1}^{n} |\lambda ^{k}| |\varphi (x^{1,k}\otimes ...\otimes
x^{m,k})|^{1-\sigma }\underset{j=1}{\overset{m}{\prod }}\Vert x^{j,k}\Vert
^{\sigma }\right\vert ^{\frac{p}{1-\sigma }}d\eta \right) ^{\frac{1-\sigma }{%
p}}.
\end{equation*}

\item[(3)] The following diagram commutes
\end{itemize}
\end{theorem}

\begin{equation*}
\xymatrix{ X_{1}\times ...\times X_{m} \ar[rr]^{T} \ar@{>}[d]_{i} & & Y\\
V_{m} \ar@{>}[rr]^{i_{p,\sigma }} & & L_{p,\sigma }(\eta ) \ ,
\ar@{>}[u]_{\hat T}}
\end{equation*}%
where $L_{p,\sigma }(\eta )$ is the domination space defined by the
seminorm%
\begin{equation*}
\left\Vert \sum_{k=1}^{n}\lambda ^{k}\left\langle x^{1,k}\otimes
...\otimes x^{m,k}, \cdot \right\rangle \right\Vert _{p,\sigma }
\end{equation*}%
\begin{equation*}
=\inf \sum_{i=1}^{r}\left( \int_{B_{(X_{1}\widehat{\otimes }_{\pi }\cdots \widehat{\otimes }_{\pi }X_{m})^{\ast }}}\left\vert \sum_{k=1}^{m} |\lambda _{i}^{k}| \left| \left\langle
x_{i}^{1,k}\otimes ...\otimes x_{i}^{m,k}, \cdot \right\rangle \right|^{1-\sigma }%
\underset{j=1}{\overset{m}{\prod }}\Vert x_{i}^{j,k}\Vert ^{\sigma
}\right\vert ^{\frac{p}{1-\sigma }}\,d\eta \right) ^{\frac{1-\sigma }{p}},
\end{equation*}%
and the infimum is taken over all representations of the form
\begin{equation*}
\sum_{k=1}^{n}\lambda ^{k}x^{1,k}\otimes ...\otimes
x^{m,k}=\sum_{i=1}^{r}\sum_{k=1}^{m}\lambda _{i}^{k}x_{i}^{1,k}\otimes
...\otimes x_{i}^{m,k}.
\end{equation*}

\vspace{5mm}

\textbf{Acknowledgement.} This article has been finished during the short stay of E. Dahia to the Universidad Polit\'ecnica de Valencia in May 2015.


\begin{thebibliography}{99}

\bibitem{adsp14} D. Achour, E. Dahia, P. Rueda and E.A. S\'anchez-P\'erez, Factorization of strongly $(p,\sigma)$-continuous multilinear operators,
Linear and Multilinear Algebra (12)62 (2014), 1649--1670.

\bibitem{BPR10} G. Botelho, D. Pellegrino and P. Rueda, A unified Pietsch
Domination Theorem, J. Math. Anal. Appl. 365 (2010), 269--276.

\bibitem{CD03} D. Carando and V. Dimant, On summability of bilinear
operators, Math. Nachr. 259 (2003), 3--11.

\bibitem{DAS12} E. Dahia, D. Achour and E.A. S\'anchez-P\'erez, Absolutely continuous multilinear operators, J. Math. Anal.
Appl. 397 (2013), 205--224.

\bibitem{DAh14} E. Dahia, On the tensorial representation of multi-linear ideals, Ph.D. thesis, University of Mohamed Boudiaf, M'sila (Algeria), 2014.

\bibitem{DF92} A. Defant and K. Floret, Tensor norms and operator ideals, North-Holland, Amsterdam, 1992.

\bibitem{5} J. Diestel, H. Jarchow, A. Tonge, Absolutely summing operators.
Cambridge University Press, 1995.

\bibitem{Dim03} V. Dimant, Strongly p-summing multilinear operators, J. Math. Anal. Appl.,
278 (2003), 182--193.

\bibitem{Gei84} S. Geiss, Ideale multilinearer Abbildungen, Diplomarbeit, 1985.

\bibitem{LiTz} J.\ Lindenstrauss and L.\ Tzafriri, Classical Banach Spaces II, Springer, Berlin, 1979.

\bibitem{LS93} J.A. L\'opez Molina and E.A. S\'anchez-P\'erez, Ideales de
operadores absolutamente continuos, Rev. Real Acad. Ciencias Exactas,
Fisicas y Naturales, Madrid 87(1993), 349--378.

\bibitem{10} M. C. Matos, On multilinear mappings of nuclear type., Rev.
Mat. Comput. 6 (1993), 61--81.

\bibitem{Matt87} U. Matter, Absolute continuous operators and
super-reflexivity, Math. Nachr, 130 (1987), 193--216.

\bibitem{oksan} S. Okada, W. J. Ricker and E. A. S\'anchez P\'{e}rez, Optimal
Domain and Integral Extension of Operators acting in Function Spaces,
Operator Theory: Adv. Appl., vol. 180, Birkhauser, Basel, 2008.

\bibitem{PeRuSa} D. Pellegrino,  P. Rueda and E. A. S\'anchez P\'erez, Surveying the spirit of absolute summability
on multilinear operators and homogeneous polynomials. To appear in RACSAM. DOI 10.1007/s13398-015-0224-8.


\bibitem{PeSa} D. Pellegrino and J. Santos, A general Pietsch Domination
Theorem, J. Math. Anal. Appl. 375 (2011), 371--374.

\bibitem{PeSaSe} D. Pellegrino, J. Santos and J. Seoane-Sep\'{u}lveda. Some techniques on nonlinear analysis and applications. Adv. Math. (2)229 (2012), 1235--1265.

\bibitem{Pie67} A. Pietsch, Absolute p-summierende Abbildungen in normierten R\"{a}umen, Studia Math., 28 (1967), 333--353.

\bibitem{13} A. Pietsch, Operator Ideals. Deutsch. Verlag Wiss., Berlin,
1978; North-Holland, Amsterdam-London-New York-Tokyo, 1980.

\bibitem{14} A. Pietsch, Ideals of multilinear functionals (designs of a
theory). Proceedings of the Second International Conference on Operator
Algebras, Ideals, and their Applications in Theoretical Physics (Leibzig).
Teubner-Texte (1983), 185--199.



\bibitem{RuSa} P. Rueda, E. A. S\'anchez P\'erez, Factorization of $p$-dominated polynomials through $L^{ p }$-spaces. Michigan Math. J. (2)63 (2014), 345--354.

\bibitem{Rya02} R. Ryan, Introduction to Tensor Product of Banach Spaces, Springer-Verlag, London, 2002.

\bibitem{YaAcRu} R. Yahi, D. Achour and P. Rueda, Absolutely summing Lipschitz conjugates and $(p,\sigma)$-summability. Preprint.

\end{thebibliography}
\end{document}